\documentclass[11pt,a4paper,reqno]{amsart}

\usepackage{amsmath,amsthm}
\usepackage{amssymb}

\usepackage{cases}

\theoremstyle{theorem}
\newtheorem{theorem}{Theorem}

\theoremstyle{definition}

\newtheorem{example}[theorem]{Example}

\theoremstyle{remark}
\newtheorem{remark}[theorem]{Remark}

\usepackage{url}
\usepackage[colorlinks,linkcolor=blue,citecolor=blue]{hyperref}

\title{Sums of powers of integers and hyperharmonic numbers}

\author[J.L. Cereceda]{Jos\'e Luis Cereceda}
\address{%
        Collado Villalba, 28400 -- Madrid, Spain}
\email{jl.cereceda@movistar.es}

\begin{document}

\begin{abstract}
In this paper, we obtain a new formula for the sums of $k$-th powers of the first $n$ positive integers, $S_k(n)$, that involves the hyperharmonic numbers and the Stirling numbers of the second kind. Then, using an explicit representation for the hyperharmonic numbers, we generalize this formula to the sums of powers of an arbitrary arithmetic progression. Furthermore, we express the Bernoulli polynomials in terms of hyperharmonic polynomials and Stirling numbers of the second kind. Finally, we extend the obtained formula for $S_k(n)$ to negative values of $n$.
\end{abstract}

\maketitle

\section{Introduction}

The $n$-th hyperharmonic number of order $r$, $H_{n}^{(r)}$, is defined recursively as
\begin{equation*}
H_{n}^{(r)} = \sum_{k=1}^{n} H_{k}^{(r-1)} \,\,\text{and}\,\, H_{n}^{(1)} = H_n,
\end{equation*}
where $H_n$ is the ordinary harmonic number $1 + \frac{1}{2} + \cdots + \frac{1}{n}$. Here it is understood that $H_{n}^{(0)}= \frac{1}{n}$ for $n \geq 1$, and $H_0^{(r)} =0$ for $r \geq 0$. In 1996, Conway and Guy \cite[p. 258]{guy} provided the following identity
\begin{equation}\label{guy}
H_{n}^{(r)} = \binom{n+r-1}{r-1} \big( H_{n+r-1} - H_{r-1} \big),
\end{equation}
connecting the hyperharmonic numbers with the harmonic numbers. There exist various proofs of formula \eqref{guy} in the literature; see, e.g., \cite{benjamin,cere2} and \cite[pp. 227--229]{mezo}.

In this paper, we obtain a new formula for the sums of powers of the first $n$ positive integers, $S_k(n) = 1^k + 2^k + \cdots + n^k$, in terms of hyperharmonic numbers and Stirling numbers of the second kind $\genfrac{\{}{\}}{0pt}{}{k}{j}$. Specifically, in Section 2, we establish the following theorem.

\begin{theorem}\label{th:1}
For any integers $k \geq 0$ and $n \geq 1$, we have
\begin{equation}\label{th1}
S_{k}(n) = \frac{(-1)^{k+1}}{k+1} \sum_{j=1}^{k+1} (-1)^{j} j! \genfrac{\{}{\}}{0pt}{}{k+1}{j} \left( H_{j+1}^{(n)} -
\frac{1}{j+1} \right).
\end{equation}
\end{theorem}

As will become clear below, $H_{j+1}^{(n)}$ is a polynomial in $n$ of degree $j$. Then, in Section 3, using the explicit representation for the hyperharmonic numbers given in \eqref{frac}, we generalize formula \eqref{th1} to the sums of powers of the terms of an arithmetic progression with first term $r$ and common difference $m$
\begin{equation}\label{prog}
S_{k}^{r,m}(n) = \sum_{j=1}^{n} ( r + (j-1) m )^k,
\end{equation}
where $m$ and $r$ are assumed to be integer variables with $m \geq 1$ and $r \geq 0$. Furthermore, in Section 4, we express the Bernoulli polynomials $B_k(x)$ in terms of the hyperharmonic polynomials $\mathcal{H}_j(x):= H_{j+1}^{(x-1)}$, $j =0,1,\ldots,k$, and the Stirling numbers of the second kind (see equation \eqref{con2}). Finally, in Section 5, we extend the above formula in equation \eqref{th1} to negative values of $n$ by making use of the definition of hyperharmonic numbers of negative order set forth by Dil and Muniro\u{g}lu in \cite{dil3}.

\section{Proof of Theorem \ref{th:1}}

Next, we detail the proof of Theorem \ref{th:1}.

\begin{proof}
We start with the following polynomial formula for $S_k(n)$; see, e.g., \cite[Equation (7.5)]{gould2}:
\begin{equation}\label{cere}
S_k(n) = \sum_{j=1}^{k} a_{k,j} \binom{n+j}{j+1},  \quad k \geq 1,
\end{equation}
where the coefficients $a_{k,j}$ are given by
\begin{equation}\label{coeff}
a_{k,j} = (-1)^{k-j} j! \genfrac{\{}{\}}{0pt}{}{k}{j}.
\end{equation}

For convenience, we can think of $n$ as being a continuous variable. This is justified by the fact that, as is well-known, for fixed $k$ there is exactly one polynomial $S_k(x)$ in $x$ such that $S_k(x) = 1^k + 2^k + \cdots + x^k$ whenever $x$ is a positive integer (see, e.g., \cite[Theorem 1]{levy} and \cite{owens}). Keeping this in mind, we invoke the following elementary result according to which (see, e.g., \cite{owens,sher,wu}):
\begin{equation}\label{deriv}
S_{k}^{\prime}(n) = \frac{\text{d}S_k(n)}{\text{d}n} = k S_{k-1}(n) + (-1)^k B_k, \quad k \geq 1,
\end{equation}
where the $B_k$'s are the Bernoulli numbers \cite{apostol}. Thus, recalling the rule for the derivative of a product of functions $h_1(x), h_2(x),
\ldots, h_j(x)$,
\begin{equation*}
\frac{\text{d}}{\text{d}x} \left( \prod_{i=1}^{j} h_i(x) \right) = \left( \prod_{i=1}^{j} h_i(x) \right) \left( \sum_{i=1}^{j}
\frac{h_{i}^{\prime}(x)}{h_i(x)} \right),
\end{equation*}
and differentiating each side of equation \eqref{cere} with respect to $n$, we obtain
\begin{align*}
S_{k}^{\prime}(n) & = \frac{\text{d}}{\text{d}n} \left( \sum_{j=1}^{k} a_{k,j}\binom{n+j}{j+1} \right)  \\
& = \sum_{j=1}^{k} \frac{a_{k,j}}{(j+1)!} \frac{\text{d}}{\text{d}n} \left( \prod_{i=0}^{j} (n+i) \right)  \\
& = \sum_{j=1}^{k} \frac{a_{k,j}}{(j+1)!} \left( \prod_{i=0}^{j} (n+i) \right) \sum_{i=0}^{j} \frac{1}{n+i} \\
& = \sum_{j=1}^{k} a_{k,j} \binom{n+j}{j+1} \big( H_{n+j} - H_{n-1} \big).
\end{align*}
By virtue of identity \eqref{guy}, the last equation reduces to
\begin{equation}\label{pr2}
S_{k}^{\prime}(n) = \sum_{j=1}^{k} a_{k,j} H_{j+1}^{(n)}, \quad k \geq 1.
\end{equation}
Now, combining equations \eqref{pr2}, \eqref{deriv}, and \eqref{coeff}, and renaming the index $k$ as $k+1$, yields
\begin{equation}\label{pr3}
S_{k}(n) = \frac{(-1)^{k+1}}{k+1} \left( \sum_{j=1}^{k+1} (-1)^{j} j! \genfrac{\{}{\}}{0pt}{}{k+1}{j} H_{j+1}^{(n)} -
B_{k+1} \right), \quad k \geq 0.
\end{equation}
On the other hand, since $S_{k}(0) =0$, and noting that $H_{j+1}^{(0)} = \frac{1}{j+1}$, from \eqref{pr3} we deduce that
\begin{equation}\label{pr4}
B_{k+1} =  \sum_{j=1}^{k+1} (-1)^j \frac{j!}{j+1} \genfrac{\{}{\}}{0pt}{}{k+1}{j}.
\end{equation}
Therefore, from \eqref{pr4} and \eqref{pr3}, we finally get \eqref{th1}.
\end{proof}

\begin{remark}
Let $D_x$ be the derivative operator with respect to $x$, i.e., $D_x f(x) = \frac{\text{d}}{\text{d}x} f(x)$. The above proof of Theorem \ref{th:1} involves essentially an application of the formula giving the hyperharmonic number $H_n^{(r)}$ as the derivative of a binomial coefficient, namely (see \cite[Section 3]{cere2} and \cite[Proposition 11]{dil3})
\begin{equation*}
D_x \left. \binom{x+n+r-1}{n} \right|_{x=0} = H_n^{(r)},
\end{equation*}
which is in turn a generalization of the equation \cite[Equation (8)]{paule} (see also \cite[Equation (Z.60)]{gould})
\begin{equation*}
D_x \left. \binom{x+n}{n} \right|_{x=0} = H_n.
\end{equation*}
\end{remark}

\begin{remark}
Equation \eqref{pr4} is a well-known property of the Bernoulli numbers (see, e.g., \cite{qi}).
\end{remark}

\begin{remark}
Letting $n =1$ in equation \eqref{pr3} yields the identity
\begin{equation*}
B_k = (-1)^{k+1} k + \sum_{j=0}^{k} (-1)^j j! \genfrac{\{}{\}}{0pt}{}{k}{j} H_{j+1}, \quad k \geq 0.
\end{equation*}
\end{remark}

As a simple example illustrating Theorem \ref{th:1}, we may use equation \eqref{th1} to calculate $S_3(n)$.
\begin{example}
For $k =3$, equation \eqref{th1} reads as
\begin{align}\label{exam}
S_3(n) & = \frac{1}{4} \sum_{j=1}^{4} (-1)^j j! \genfrac{\{}{\}}{0pt}{}{4}{j} \left( H_{j+1}^{(n)} - \frac{1}{j+1} \right)  \notag \\
& = 6 H_{5}^{(n)} - 9 H_{4}^{(n)} + \frac{7}{2} H_{3}^{(n)} - \frac{1}{4} H_{2}^{(n)} + \frac{1}{120}.
\end{align}
In order to evaluate the involved hyperharmonic numbers $H_{j}^{(n)}$, it is useful to employ the following explicit formula derived in \cite[Theorem 1]{benjamin} and, additionally, in \cite[Theorem 5]{dil}
\begin{equation}\label{frac}
H_{j}^{(n)} = \sum_{t=1}^{j} \binom{n+j-t-1}{j-t} \frac{1}{t}, \quad n,j \geq 1,
\end{equation}
which gives $H_{j}^{(n)}$ as a weighted sum of the fractions $\frac{1}{1}, \frac{1}{2}, \ldots, \frac{1}{j}$. From \eqref{frac}, it is easily seen that $H_{j}^{(n)}$ is a polynomial in $n$ of degree $j-1$ with leading coefficient $\frac{1}{(j-1)!}$ and constant term $\frac{1}{j}$. Applying \eqref{frac}, we obtain
\begin{align*}
H_{2}^{(n)} & = n + \tfrac{1}{2},  \\
H_{3}^{(n)} & =\tfrac{1}{2}n^2 + n + \tfrac{1}{3},  \\
H_{4}^{(n)} & = \tfrac{1}{6}n^3 + \tfrac{3}{4}n^2 + \tfrac{11}{12}n + \tfrac{1}{4}, \\
H_{5}^{(n)} & = \tfrac{1}{24}n^4 +  \tfrac{1}{3}n^3 + \tfrac{7}{8}n^2 + \tfrac{5}{6}n + \tfrac{1}{5}.
\end{align*}
Substituting these expressions into \eqref{exam} and simplifying, we find that, as expected, $S_3(n) = \tfrac{1}{4} n^2 (n+1)^2$.
\end{example}

\begin{remark}
By using equation \eqref{frac} into \eqref{th1}, we can equivalently express $S_k(n)$ as a weighted sum of $\frac{1}{1}, \frac{1}{2}, \ldots,
\frac{1}{k+1}$ as follows
\begin{equation}\label{rem1}
S_k(n) = \frac{(-1)^{k+1}}{k+1} \sum_{t=1}^{k+1} V_{k,t}(n) \frac{1}{t},  \quad k \geq 0,
\end{equation}
where
\begin{equation}\label{rem2}
V_{k,t}(n) = \sum_{i=t}^{k+1} (-1)^i i! \genfrac{\{}{\}}{0pt}{}{k+1}{i} \binom{n+i-t}{i+1-t}.
\end{equation}
\end{remark}

\begin{remark}
As shown in \cite[Theorem 2]{benjamin}, the hyperharmonic numbers, $H_{n}^{(r)}$, and the $r$-Stirling numbers of the first kind,
$\genfrac{[}{]}{0pt}{}{n}{k}_r$, are related by
\begin{equation*}\label{benj}
H_{n}^{(r)} = \frac{1}{n!} \genfrac{[}{]}{0pt}{}{n+r}{r+1}_r.
\end{equation*}
Therefore, from Theorem \ref{th:1}, we can alternatively write $S_k(n)$ in the form
\begin{equation}\label{r1}
S_{k}(n) = \frac{(-1)^{k+1}}{k+1} \sum_{j=1}^{k+1} \frac{(-1)^{j}}{j+1} \genfrac{\{}{\}}{0pt}{}{k+1}{j}
\left( \genfrac{[}{]}{0pt}{}{n+j+1}{n+1}_n -  j! \right).
\end{equation}
This formula is to be complemented by the following one
\begin{equation}\label{r2}
\genfrac{[}{]}{0pt}{}{n+j+1}{n+1}_n = \sum_{i=1}^{j+1} (i-1)! \binom{j+1}{i} n^{\overline{j+1-i}},
\end{equation}
expressing $\genfrac{[}{]}{0pt}{}{n+j+1}{n+1}_n$ in terms of the rising factorials $n^{\overline{j+1-i}}$, $i =1,2,\ldots,j+1$. Hence,
substituting \eqref{r2} into \eqref{r1}, we get
\begin{equation*}
S_{k}(n) = \frac{(-1)^{k+1}}{k+1} \sum_{j=1}^{k+1} \frac{(-1)^{j}}{j+1} \genfrac{\{}{\}}{0pt}{}{k+1}{j}
 \sum_{i=1}^{j} (i-1)! \binom{j+1}{i} n^{\overline{j+1-i}}.
\end{equation*}
\end{remark}

\section{Generalization of Theorem \ref{th:1}}

Next, using the representation for the hyperharmonic numbers given in \eqref{frac}, we generalize Theorem \ref{th:1} to the arithmetic progression defined in \eqref{prog}.

\begin{theorem}
For any integers $k \geq 0$ and $n \geq 1$, we have
\begin{equation}\label{th2}
S_{k}^{r,m}(n) = (-1)^{k+1} \frac{m^k}{k+1} \sum_{t=1}^{k+1} \left( V_{k,t} \left(n-1+\frac{r}{m}\right) -
V_{k,t} \left(\frac{r}{m}-1 \right) \right) \frac{1}{t},
\end{equation}
where $V_{k,t}(n)$ is the polynomial given in \eqref{rem2}.
\end{theorem}
\begin{proof}
This follows in a rather straightforward way from the following simple but powerful result derived in \cite{griff}. Let $S_k(x)$ denote the unique polynomial in $x$ such that, for all $n \geq 1$, $S_k(n)$ gives us the sum of powers of the first $n$ positive integers (with $S_k(0) =0$). Then,
for any real number $x$, it turns out that \cite{griff}
\begin{equation}\label{th3}
\sum_{j=1}^{n} (j +x)^k = S_k(n+x) - S_k(x).
\end{equation}
Taking $x = \frac{r}{m} -1$ in \eqref{th3} yields
\begin{equation*}
S_{k}^{r,m}(n) = m^k \left( S_{k} \left(n-1+\frac{r}{m}\right) - S_{k}\left(\frac{r}{m}-1 \right) \right).
\end{equation*}
Hence, using the polynomial formula for $S_k(n)$ given in \eqref{rem1}, we get \eqref{th2}.
\end{proof}

\begin{remark}
Equation \eqref{th2} reduces to \eqref{rem1} when $r =m=1$.
\end{remark}

\section{Bernoulli polynomials}

In this section we derive an expression for the Bernoulli polynomials $B_k(x)$ involving the hyperharmonic polynomials $\mathcal{H}_j(x):= H_{j+1}^{(x-1)}$, $j =0,1,\ldots,k$, and the Stirling numbers of the second kind. Starting from the well-known relationship between $S_k(n)$ and $B_k(n)$, namely
\begin{equation*}
S_k (n) = \frac{1}{k+1} \big( B_{k+1}(n+1) - B_{k+1} \big), \quad k \geq 1,
\end{equation*}
it follows that $B_{k+1}^{\prime}(n+1) = (k+1) S_{k}^{\prime}(n)$. On the other hand, we have that \cite{apostol} $B_{k+1}^
{\prime}(n+1) = (k+1)B_{k}(n+1)$. Therefore, from \eqref{pr2}, we obtain that
\begin{equation}\label{con1}
B_{k}(n+1) = \sum_{j=0}^{k} a_{k,j} H_{j+1}^{(n)}, \quad k \geq 0.
\end{equation}
Clearly, the right-hand side of \eqref{con1} is a polynomial in $n$ of degree $k$. Hence, using equations \eqref{coeff}, \eqref{frac}, and \eqref{con1}, one can naturally extend $B_{k}(n+1)$ to a polynomial $B_k(x)$ in which $x$ takes any real value as follows
\begin{equation}\label{con2}
B_k(x) = \sum_{j=0}^{k} (-1)^{k-j} j! \genfrac{\{}{\}}{0pt}{}{k}{j} \mathcal{H}_{j}(x),  \quad k \geq 0,
\end{equation}
where
\begin{equation*}
\mathcal{H}_{j}(x):= H_{j+1}^{(x-1)} = \frac{1}{j} \left( x -  \frac{1}{j+1} \right) + \sum_{t=1}^{j-1}
\binom{x+j-t-1}{j+1-t} \frac{1}{t}, \quad j \geq 2,
\end{equation*}
with $\mathcal{H}_{0}(x) =1$ and $\mathcal{H}_{1}(x):= H_{2}^{(x-1)} = x - \frac{1}{2}$.

Note that, since $\mathcal{H}_{j}(0) = -\frac{1}{j(j+1)}$ for all $j \geq 1$, it follows from \eqref{con2} that
\begin{equation*}
B_k = (-1)^{k+1} \sum_{j=1}^{k} (-1)^j \frac{(j-1)!}{j+1}\genfrac{\{}{\}}{0pt}{}{k}{j}, \quad k \geq 1,
\end{equation*}
which is a variant of the identity in \eqref{pr4}. This formula for the Bernoulli numbers has recently been derived in \cite[Equation (3)]{jha}. On the other hand, using \eqref{con1} and the difference equation \cite{apostol}, $B_{k}(x+1) - B_{k}(x) = k x^{k-1}$, we obtain the following alternative
formula for $B_k(x)$:
\begin{equation*}\label{kx}
B_k(x) = \sum_{j=0}^{k} (-1)^{k-j} j! \genfrac{\{}{\}}{0pt}{}{k}{j} H_{j+1}^{(x)} - k x^{k-1},  \quad k \geq 0,
\end{equation*}
where
\begin{equation*}
H_{j+1}^{(x)} = \frac{1}{j+1} + \sum_{t=1}^{j} \binom{x+j-t}{j+1-t}\frac{1}{t}, \quad j \geq 1,
\end{equation*}
and $H_{1}^{(x)} =1$.

Let us further note that we can reverse \eqref{con2} to obtain
\begin{equation}\label{conv}
\mathcal{H}_k(x) = \frac{1}{k!} \sum_{j=0}^{k} \genfrac{[}{]}{0pt}{}{k}{j} B_j(x), \quad k \geq 0,
\end{equation}
where $\genfrac{[}{]}{0pt}{}{k}{j}$ are the (unsigned) Stirling numbers of the first kind. Incidentally, setting $x=0$ in \eqref{conv} allows
us to deduce the following recursive formula for the Bernoulli numbers:
\begin{equation*}
\sum_{j=1}^{k} \genfrac{[}{]}{0pt}{}{k}{j} B_j = - \frac{(k-1)!}{k+1}, \quad k \geq 1.
\end{equation*}
A proof of this last identity using the Riordan array method can be found in \cite[p. 288]{sprugnoli}.

We end this section with the following important observation.

\begin{remark}
The $j$-th degree polynomials $\mathcal{H}_j(x):= H_{j+1}^{(x-1)}$ introduced in this paper are closely related to the so-called {\it harmonic polynomials\/} $H_j(x)$ of degree $j$ in $x$ defined in \cite[Equation (28)]{cheon} by the ordinary generating function
\begin{equation*}
\frac{-\ln (1-t)}{t (1-t)^{1-x}} = \sum_{j=0}^{\infty} H_j(x) t^j,
\end{equation*}
where $H_j(0) = H_{j+1}$. Indeed, it turns out that
\begin{align*}
H_j(x) & =  H_{j+1}^{(1-x)} \quad \text{and}\quad H_j(x) = \mathcal{H}_j(2-x),
\intertext{or, conversely,}
H_{j+1}^{(x)} & =  H_j(1-x) \quad \text{and}\quad \mathcal{H}_j(x) = H_j(2-x).
\end{align*}
The harmonic polynomials $H_j(x)$, $j \geq 0$, have, in particular, the explicit representation (see \cite[Theorem 5.4]{cheon})
\begin{equation*}
H_j(x) = \sum_{t=1}^{j+1} \binom{j+1-t-x}{j+1-t} \frac{1}{t},
\end{equation*}
which can be recovered by letting $j \to j+1$ and $n \to 1-x$ in \eqref{frac}.
\end{remark}

\section{Extension of formula \eqref{th1} to negative values of $n$}

In \cite{mezo2}, Mez\H{o} defined the hyperharmonic function $H_z^{(w)}$ involving the Pochhammer symbol $(z)_w$, gamma $\Gamma(w)$ and digamma $\Psi(w)$ functions, as
\begin{equation*}
H_z^{(w)} = \frac{(z)_w}{z \Gamma(w)} \big( \Psi(z +w) - \Psi(w) \big),
\end{equation*}
where $w, z+w \in \mathbb{C} \setminus \mathbb{Z}^{-}$, and $\mathbb{Z}^{-} = \{0, -1, -2, \ldots\, \}$. Based on the hyperharmonic function, Dil \cite{dil2} presented formulas to calculate special values of $H_z^{(w)}$ subjected to the above restriction of $w, z+w \in \mathbb{C} \setminus \mathbb{Z}^{-}$. Subsequently, Dil and Muniro\u{g}lu \cite{dil3} showed a way to define ``negative-ordered hyperharmonic numbers''. According to \cite[Definition 25]{dil3}, for positive integers $n$ and $r$, the hyperharmonic number of negative order $H_n^{(-r)}$ can be defined by
\begin{equation}\label{def}
H_n^{(-r)} =\left\{
              \begin{array}{ll}
                \dfrac{(-1)^r r!}{n^{\underline{r+1}}}, & n >r \geq 1; \\[4mm]
                \displaystyle\sum_{i=0}^{n-1} (-1)^i \binom{r}{i} \dfrac{1}{n-i}, & r \geq n >1; \\[5mm]
                1, & n=1.
              \end{array}
            \right.
\end{equation}
Furthermore, as noted in \cite{dil3}, the identity (see \cite[Equation (1.43)]{gould2})
\begin{equation*}
\sum_{i=0}^{r} (-1)^i \binom{r}{i} \frac{1}{n-i} = \frac{(-1)^r}{(n-r)\binom{n}{r}},
\end{equation*}
ensures the consistency of the definition in \eqref{def}.

Therefore, for $n \geq 1$, we can use \eqref{def} to define $H_{j+1}^{(-n)}$ as follows
\begin{equation}\label{def2}
H_{j+1}^{(-n)} =
\left\{
  \begin{array}{ll}
   \displaystyle\sum_{i=0}^{j} \frac{(-1)^i}{j+1-i} \binom{n}{i}, &  j \geq 1 ; \\
    1, & j=0,
  \end{array}
\right.
\end{equation}
so that the extension of formula \eqref{th1} to negative values of $n$ can effectively be stated as
\begin{equation}\label{ext}
S_{k}(-n) = \frac{(-1)^{k+1}}{k+1} \sum_{j=1}^{k+1} (-1)^{j} j! \genfrac{\{}{\}}{0pt}{}{k+1}{j}
\sum_{i=1}^{j} \frac{(-1)^i}{j+1-i} \binom{n}{i},
\end{equation}
for $k \geq 0$ and $n \geq 1$.

\begin{remark}
It is a well-known fact that $S_1(n)$ is a factor of $S_k(n)$ for all $k \geq 1$, which means that $S_k(-1) =0$ for all $k\geq 1$. We can check from \eqref{ext} that the latter holds true. Indeed, setting $n=1$ in \eqref{ext} gives
\begin{equation*}
S_k(-1) = \frac{(-1)^{k}}{k+1} \sum_{j=1}^{k+1} (-1)^{j} (j-1)! \genfrac{\{}{\}}{0pt}{}{k+1}{j},
\end{equation*}
which is identically equal to zero for $k \geq 1$, according to the identity (A.17) in \cite{boya}.
\end{remark}

\begin{remark}
As we saw in the preceding section the harmonic and hyperharmonic polynomials are related by $H_j(x) = H_{j+1}^{(1-x)}$ or, $H_j(x) = H_{j+1}^{(-(x-1))}$. Since the definition given in \eqref{def2} is valid for any $n \geq 1$, we can licitly use \eqref{def2} to obtain the following representation of the harmonic polynomials introduced in \cite[Section 5]{cheon}
\begin{equation*}
H_j(x) = \sum_{i=0}^{j} \frac{(-1)^i}{j+1-i} \binom{x-1}{i}, \quad j \geq 1,
\end{equation*}
and $H_0(x) =1$. In particular, since $\binom{-1}{i} = (-1)^i$, the last formula yields, as it should be, $H_j(0) = H_{j+1}$.
\end{remark}

To close this paper, it is worth noting the symmetry property of the power sum polynomials $S_k(n)$, namely \cite{newsome}
\begin{equation}\label{sym}
S_k( -(n+1)) = (-1)^{k+1} S_k(n), \quad k \geq 1.
\end{equation}
Thus, using \eqref{th1}, \eqref{ext}, and \eqref{sym}, we can express $S_k(n)$ in the alternative form
\begin{equation*}
S_k(n) = \frac{1}{k+1} \sum_{j=1}^{k+1} (-1)^{j} j! \genfrac{\{}{\}}{0pt}{}{k+1}{j}
\sum_{i=1}^{j} \frac{(-1)^i}{j+1-i} \binom{n+1}{i},
\end{equation*}
for $k \geq 1$ and $n \geq 0$.

\section*{Acknowledgements}

The author is grateful to the anonymous referees for their valuable comments and suggestions which led to improvements to this paper. He would also like to thank one of the referees for providing some pertinent references.


\vspace{2mm}

\begin{thebibliography}{99}


\bibitem{apostol} Apostol, T. M. (2008). A primer on Bernoulli numbers and polynomials. {\it Mathematics Magazine}, 81(3), 178--190.


\bibitem{benjamin} Benjamin, A. T., Gaebler, D., \& Gaebler, R. (2003). A combinatorial approach to hyperharmonic numbers.
{\it Integers}, 3, 1--9. Article \#A15.


\bibitem{boya} Boyadzhiev, K. N. (2018). {\it Notes on the Binomial Transform: Theory and Table with Appendix on Stirling Transform}. World Scientific. Singapore.


\bibitem{cere2} Cereceda, J. L. (2015). An introduction to hyperharmonic numbers. {\it International Journal of Mathematical Education in Science and Technology}, 46(3), 461--469.


\bibitem{cheon} Cheon, Gi.-S., \& El-Mikkawy, M. E. A. (2008). Generalized harmonic numbers with Riordan arrays. {\it Journal of Number Theory}, 128(2), 413--425.


\bibitem{guy} Conway, J. H., \& Guy, R. K. (1996). {\it The Book of Numbers}. Copernicus. New York.


\bibitem{dil} Dil, A., \& Mez\H{o}, I. (2008). A symmetric algorithm for hyperharmonic and Fibonacci numbers. {\it Applied Mathematics and Computation}, 206(2), 942--951.


\bibitem{dil2} Dil, A. (2019). On the hyperharmonic function. {\it S\"{u}leyman Demirel University, Journal of Natural and Applied Sciences}, 23, Special Issue, 187--193.


\bibitem{dil3} Dil, A., \& Muniro\v{g}lu, E. (2020). Applications of derivative and difference operators on some sequences. {\it Applicable Analysis and Discrete Mathematics}, 14(2), 406--430.


\bibitem{gould} Gould, H. W. (1972). {\it Combinatorial Identities: A Standardized Set of Tables Listing 500 Binomial Coefficient Summations}. Morgantown Printing and Binding Co., Morgantown (WV).


\bibitem{gould2} Gould, H. W. (1978). Evaluation of sums of convolved powers using Stirling and Eulerian numbers. {\it The Fibonacci Quarterly}, 16(6), 488--497.


\bibitem{griff} Griffiths, M. (2002). Sums of powers of the terms in any finite arithmetic progression. {\it Mathematical Gazette}, 86(506), 269--271.


\bibitem{jha} Jha, S. K. (2020). Two new explicit formulas for the Bernoulli numbers. {\it Integers}, 20, Article \#A21, 5 pp.


\bibitem{levy} Levy, L. S. (1970). Summation of the series $1^n + 2^n + \cdots + x^n$ using elementary calculus. {\it American Mathematical Monthly}, 77(8), 840--847.


\bibitem{mezo2} Mez\H{o}, I. (2009). Analytic extension of hyperharmonic numbers. {\it Online Journal of Analytic Combinatorics},
4, Article 1, 9 pp.


\bibitem{mezo} Mez\H{o}, I. (2020). {\it Combinatorics and Number Theory of Counting Sequences}. CRC Press. Taylor \& Francis Group. Boca Raton (FL).


\bibitem{newsome} Newsome, N. J., Nogin, M. S., \& Sabuwala, A. H. (2017). A proof of symmetry of the power sum polynomials using a novel Bernoulli number identity. {\it Journal of Integer Sequences}, 20, Article 17.6.6, 10 pp.


\bibitem{owens} Owens, R. W. (1992). Sums of powers of integers. {\it Mathematics Magazine}, 65(1), 38--40.


\bibitem{paule} Paule, P., \& Schneider, C. (2003). Computer proofs of a new family of harmonic number identities. {\it Advances in Applied Mathematics}, 31(2), 359--378.


\bibitem{qi} Qi, F., \& Guo, B.-N. (2014). Alternative proofs of a formula for Bernoulli numbers in terms of Stirling numbers.
{\it Analysis}, 34(3), 311--317.


\bibitem{sher} Sherwood, H. (1970). Sums of powers of integers and Bernoulli numbers. {\it Mathematical Gazette}, 54(389), 272--274.


\bibitem{sprugnoli} Sprugnoli, R. (1994). Riordan arrays and combinatorial sums. {\it Discrete Mathematics}, 132(1--3), 267--290.


\bibitem{wu} Wu, D. W. (2001). Bernoulli numbers and sums of powers. {\it International Journal of Mathematical Education in
Science and Technology}, 32(3), 440--443.


\end{thebibliography}
\end{document}